\documentclass[12pt,reqno]{amsart}
\usepackage{amsmath,amssymb,amsfonts,amsthm}
\usepackage[left=3cm, right=3cm, top=2.8cm, bottom=2.8cm]{geometry}
\usepackage[utf8]{inputenc}
\usepackage[T1]{fontenc}
\usepackage{color}
\usepackage{multirow}
\usepackage{graphicx}
\usepackage{hyperref}
\usepackage{graphicx}
\usepackage{calc}%
\usepackage{hyperref}
\hypersetup{
    colorlinks=true,
    linkcolor=blue,
    filecolor=magenta,
    urlcolor=cyan,
}

\usepackage{stackrel}

\DeclareMathOperator*{\esssup}{ess\,sup}
\newtheorem{theorem}{Theorem}
\newtheorem{lemma}[theorem]{Lemma}
\newtheorem{proposition}[theorem]{Proposition}
\newtheorem{corollary}[theorem]{Corollary}
\newtheorem{definition}[theorem]{Definition}
\newtheorem{remark}[theorem]{Remark}

\theoremstyle{plain}

\begin{document}

\title[The Riemann-Liouville fractional integral in Bochner-Lebesgue spaces III]{The Riemann-Liouville fractional integral in Bochner-Lebesgue spaces III}


\author[P. M. Carvalho-Neto]{Paulo M. de Carvalho-Neto}
\address[Paulo M. de Carvalho Neto]{Departamento de Matem\'atica, Centro de Ciências Físicas e Matemáticas, Universidade Federal de Santa Catarina, Florian\'{o}polis - SC, Brazil}
\email[]{paulo.carvalho@ufsc.br}
\author[R. Fehlberg J\'{u}nior]{Renato Fehlberg J\'{u}nior}
\address[Renato Fehlberg Junior]{Departamento de Matem\'atica, Universidade Federal do Esp\'{i}rito Santo, Vit\'{o}ria - ES, Brazil}
\email[]{renato.fehlberg@ufes.br}


\subjclass[2010]{26A33, 47G10, 30H35}


\keywords{Riemann-Liouville fractional integral, BMO space,  Hardy-Littlewood Theorem, Bochner-Lebesgue space}


\begin{abstract}
In this manuscript, we examine the continuity properties of the Riemann-Liouville fractional integral for order $\alpha = 1/p$, where $p > 1$, mapping from $L^p(t_0, t_1; X)$ to the Banach space $BMO(t_0, t_1; X)\cap K_{(p-1)/p}(t_0, t_1; X)$. This improvement, in some sense, refines a result by Hardy-Littlewood (\cite{HaLi1}). To achieve this, we study properties between spaces $BMO(t_0, t_1; X)$ and $K_{(p-1)/p}(t_0, t_1; X)$. Additionally, we obtained the boundedness of the fractional integral of order $\alpha \geq 1$ from $L^1(t_0, t_1; X)$ into the Riemann-Liouville fractional Sobolev space $W^{s,p}_{RL}(t_0, t_1; X)$.
\end{abstract}

\maketitle

\section{Introduction}

Historically, Hardy and Littlewood (see \cite{HaLi1} for details) were among the first to initiate discussions regarding the boundedness of the Riemann-Liouville fractional integral of order $\alpha > 0$ on the Lebesgue spaces $L^p(t_0, t_1)$. They managed to obtain several interesting results that laid the foundation for this theory in Harmonic Analysis.

In recent years, the authors have begun to study the Riemann-Liouville fractional integral of order $\alpha > 0$ as an operator over the Bochner-Lebesgue spaces (see \cite{CarFe0,CarFe1}). This exploration was motivated by the existing gaps in the literature, which sometimes posed obstacles in other studies involving the theory of fractional calculus. In the author's previous work \cite{CarFe1}, they studied the Riemann-Liouville fractional integral of order $\alpha\in(0,1/p)$, with $p>1$, from $L^p(t_0,t_1;X)$ into $L^{p/(1-p\alpha)}(t_0,t_1;X)$. However, the critical case, where $p=1/\alpha$, remained unexplored in that work. Initially, one might reasonably expect the Riemann-Liouville integral of order $1/p$ maps $L^p(t_0,t_1;X)$ into $L^\infty(t_0,t_1;X)$. However, this expectation does not align with reality. While Hardy-Littlewood previously addressed this issue when $X=\mathbb{R}$, showing that the Riemann-Liouville fractional integral of order $p=1/\alpha$ from $L^p(t_0,t_1)$ into $L^{p}(t_0,t_1)$, but not maps to $L^\infty(t_0,t_1)$, their discussion lacked a complete answer, that we can summarize as: what is the space in the intersection of the $L^{p}(t_0,t_1)$ spaces in which fractional integral of order $p=1/\alpha$ would map $L^{p}(t_0,t_1)$ continuously? This is why we are now motivated to delve deeper into this matter. In this work, we establish that the Riemann-Liouville fractional integral of order $1/p$ is a bounded operator when considered from $L^p(t_0,t_1;X)$ into $BMO(t_0,t_1;X)\cap K_\gamma(t_0,t_1;X)$, for $\gamma\geq1/p^\prime$. It is worth mentioning that there are older partial references on this subject, but they lack completeness (see Subsection \ref{subboa}).

The other scenario, referred to here as the ``Regular Case'', delves into the regularity of the Riemann-Liouville fractional integral of order $\alpha \geq 1$ from $L^1(t_0, t_1; X)$ into the RL fractional Sobolev space denoted by $W^{s,p}_{RL}(t_0, t_1; X)$. This case is obtained as a consequence of the author's previous work \cite{CarFe0,CarFe1}. It is important to emphasize that the concept of non-standard Sobolev spaces has been previously explored in the literature, particularly when $X = \mathbb{R}$; see, for instance, \cite{BeLeNaTo1, CaCo1}.

To summarize the article's organization, Section 2 introduces classical concepts related to Riemann-Liouville fractional integration and differentiation of order $\alpha > 0$. In Section 3, we provid a description of $W^{s,p}_{RL}(t_0, t_1; X)$, along with a proof for the ``Regular Case''. Section 4, we describe the spaces $BMO(t_0,t_1;X)$ and $K_\gamma(t_0,t_1;X)$, given some relation between this spaces. Section $5$, we study the ``Critical Case'' that discusses the critical aspect of the problem as previously explained. In Section 6 we address some open questions.

\section{On the Riemann-Liouville Fractional Integration}

Throughout this manuscript, we assume that $t_0 < t_1$ are fixed real numbers, and $X$ represents a Banach space. To ensure the completeness of this work, we revisit the definitions and notations of some specific classical vector spaces and operators. We begin presenting the classical Bochner-Lebesgue spaces (for further details, see \cite{ArBaHiNe1}).

\begin{definition} If $1\leq p\leq\infty$, we represent the set of all Bochner measurable functions $f:[t_0,t_1]\rightarrow X$, for which $\|f(t)\|_X\in L^{p}(t_0,t_1;\mathbb{R})$, by the symbol $L^{p}(t_0,t_1;{X})$. Moreover, $L^{p}(t_0,t_1;{X})$ is a Banach space when considered with the norm
$$\|f\|_{L^p(t_0,t_1;X)}:=\left\{\begin{array}{ll}\bigg[\displaystyle\int_a^b{\|f(s)\|^p_X}\,ds\bigg]^{1/p},&\textrm{ if }p\in[1,\infty),\vspace*{0.3cm}\\
\esssup_{s\in [a,b]}\|f(s)\|_X,&\textrm{ if }p=\infty.\end{array}\right.$$
We use the symbol $L_{loc}^p(t_0, t_1; X)$, for $1\leq p\leq\infty$, to denote the set of functions that belong to $L^p(K; X)$ for any compact $K\subset (t_0, t_1)$.
\end{definition}

We also recall the Bochner-Sobolev spaces (for more details see \cite{CaHa1}).

\begin{definition} Let $k\in\mathbb{N}^*:=\{1,2,\ldots\}$ and $p\geq1$. The set $W^{k,p}(t_0,t_1;X)$ denotes the subspace of $L^{p}(t_0,t_1;X)$ of every function $f:[t_0,t_1]\rightarrow X$ that has $k-$weak derivatives in $L^{p}(t_0,t_1;X)$. By considering the norm
$$\|f\|_{W^{k,p}(t_0,t_1;X)}:=\sum_{j=0}^k\big\|f^{(j)}\big\|_{L^{p}(t_0,t_1;X)}$$
the set $W^{k,p}(t_0,t_1;X)$ becomes a Banach space. For the completeness of the definition we assume that $W^{0,p}(t_0,t_1;X)=L^p(t_0,t_1;X)$.
\end{definition}

Now we present the notions of Riemann-Liouville fractional integral and derivative.

\begin{definition} Let $\alpha\in(0,\infty)$ and $f:[t_0,t_1]\rightarrow{X}$. The Riemann-Liouville (RL for short) fractional integral of order $\alpha$ at $t_0$ of function $f$ is defined by
\begin{equation}\label{fracinit}J_{t_0,t}^\alpha f(t):=\dfrac{1}{\Gamma(\alpha)}\displaystyle\int_{t_0}^{t}{(t-s)^{\alpha-1}f(s)}\,ds,\end{equation}
for every $t\in [t_0,t_1]$ such that integral \eqref{fracinit} exists. Above $\Gamma$ denotes the classical Euler's gamma function.
\end{definition}

\begin{remark}\label{remarkfracriint} We emphasize the following properties of the RL fractional integral:
\begin{itemize}
\item[(i)] For $\alpha>0$ and $1\leq p\leq\infty$, it holds that
\begin{equation*}\left\|J_{t_0,t}^\alpha f(t)\right\|_{L^{p}(t_0,t_1;X)}\leq \left[\dfrac{(t_1-t_0)^\alpha}{\Gamma(\alpha+1)}\right] \|f\|_{L^p(t_0,t_1;X)}.\end{equation*}
In other words, $J_{t_0,t}^\alpha$ is a bounded linear operator from $L^p(t_0,t_1;X)$ into itself. For details see \cite[Theorem 3.1]{CarFe0}.\vspace*{0.2cm}
\item[(ii)] For $1\leq p<\infty$ and $f\in L^p(t_0,t_1;X)$, we have that
\begin{equation*}\lim_{\alpha\rightarrow0^+}{\big\|J_{t_0,t}^\alpha f-f\big\|_{L^p(t_0,t_1;X)}}=0.\end{equation*}
This is the reason we define $J_{t_0,t}^0 f(t):=f(t)$. For details see \cite[Theorem 3.10]{CarFe0}.\vspace*{0.2cm}
\item[(iii)] For $1\leq p<\infty$ we have that $\{J_{t_0,t}^\alpha:\alpha\geq0\}\subset\mathcal{L}(L^p(t_0,t_1;X))$ defines a $C_0$-semigroup. 
    %
    %
    For details see \cite[Theorem 3.15]{CarFe0}.
\end{itemize}
\end{remark}

\begin{definition} Assume that $\alpha\in(0,\infty)$ and consider $f:[t_0,t_1]\rightarrow{X}$. The Riemann-Liouville (RL for short) fractional derivative of order $\alpha$ at $t_0$ of function $f$ is defined by
\begin{equation}\label{fracinit}D_{t_0,t}^\alpha f(t):=\dfrac{d^{[\alpha]}}{dt^{[\alpha]}}\left[J_{t_0,t}^{[\alpha]-\alpha} f(t)\right]=\dfrac{d^{[\alpha]}}{dt^{[\alpha]}}\left[\dfrac{1}{\Gamma([\alpha]-\alpha)}\displaystyle\int_{t_0}^{t}{(t-s)^{[\alpha]-\alpha-1}f(s)}\,ds\right],\end{equation}
for every $t\in [t_0,t_1]$ such that the right side of \eqref{fracinit} exists. Above the derivative is considered in the weak sense and the symbol $[\cdot]$ is used to denote the least integer greater than $\alpha$.
\end{definition}

\begin{remark}\label{remarkparaajud}
\begin{itemize}
\item[(i)] If $1\leq p<\infty$, $k\in\mathbb{N}$ and $f\in W^{k,p}(t_0,t_1;X)$, follows from the definition of RL fractional derivative and item $(ii)$ of Remark \ref{remarkfracriint} that
$$D_{t_0,t}^kf(t)=\dfrac{d^{k}}{dt^{k}}\left[J_{t_0,t}^{0} f(t)\right]=f^{(k)}(t),$$
for a.e. $t\in[t_0,t_1]$.\vspace*{0.2cm}
\item[(ii)] By item $(iii)$ of Remark \ref{remarkfracriint}, if $0<\beta\leq\alpha$ and $f\in L^1(t_0,t_1;X)$ then
$$D_{t_0,t}^\beta \left[J_{t_0,t}^\alpha f(t)\right]=J_{t_0,t}^{\alpha-\beta}f(t),$$
for a.e. $t\in[t_0,t_1]$.
\end{itemize}
\end{remark}

\section{The Boundedness of RL Fractional Integral - Regular Case}

In this short section we improve the result cited in Remark \ref{remarkfracriint}. For this purpose, we introduce the space $W_{RL}^{\alpha,p}(t_0,t_1;X)$.

 \begin{definition}\label{sobolevriemann} Let $\alpha>0$ and $p\geq1$. The set $W_{RL}^{\alpha,p}(t_0,t_1;X)$ denotes the subspace of $W^{[\alpha]-1,p}(t_0,t_1;X)$ which is composed of every function $f:[t_0,t_1]\rightarrow X$ such that $D_{t_0,t}^\alpha f\in L^{p}(t_0,t_1;X)$.
\end{definition}

\begin{proposition} By considering the norm
\begin{equation*}\|f\|_{W_{RL}^{\alpha,p}(t_0,t_1;X)}:=\|f\|_{W^{[\alpha]-1,p}(t_0,t_1;X)}+\|D_{t_0,t}^\alpha f\|_{L^p(t_0,t_1;X)},\end{equation*}
the set $W_{RL}^{\alpha,p}(t_0,t_1;X)$ becomes a Banach space. We call $W_{RL}^{\alpha,p}(t_0,t_1;X)$ the RL fractional Sobolev spaces.
\end{proposition}
\begin{proof} If $\{\phi_j\}_{j=1}^\infty\subset W_{RL}^{\alpha,p}(t_0,t_1;X)$ is a Cauchy sequence, then there should exist $g\in W^{[\alpha]-1,p}(t_0,t_1;X)$ and $h\in L^p(t_0,t_1;X)$ such that $\phi_j\rightarrow g$ in the topology of $W^{[\alpha]-1,p}(t_0,t_1;X)$ and $D_{t_0,t}^\alpha\phi_j\rightarrow h$ in the topology of $L^p(t_0,t_1;X)$. Hence, item $(i)$ of Remark \ref{remarkfracriint} ensures that $J_{t_0,t}^{[\alpha]-\alpha}\phi_j\rightarrow J_{t_0,t}^{[\alpha]-\alpha}g$ in the topology of $L^p(t_0,t_1;X)$. Finally, the definition of weak derivative and a classical argument of convergence allows us to conclude that $D_{t_0,t}^\alpha g(t)=h(t)$ for a.e. $t\in[t_0,t_1]$, proving that $g\in W_{RL}^{\alpha,p}(t_0,t_1;X)$ and $\phi_j\rightarrow g$ in the topology of $W_{RL}^{\alpha,p}(t_0,t_1;X)$, as we wanted.
\end{proof}

\begin{remark}
For $\alpha\in\mathbb{N}$ and $p\geq1$, we have that $W_{RL}^{\alpha,p}(t_0,t_1;X)=W^{\alpha,p}(t_0,t_1;X)$ and $\|\cdot\|_{W_{RL}^{\alpha,p}(t_0,t_1;X)}=\|\cdot\|_{W^{\alpha,p}(t_0,t_1;X)}$. Just recall item $(i)$ of Remark \ref{remarkparaajud}.\vspace*{0.2cm}
\end{remark}

As a consequence of Remarks \ref{remarkfracriint} and \ref{remarkparaajud}, we obtain:

\begin{theorem} If $\alpha\in[1,\infty)$, then $J_{t_0,t}^\alpha:L^1(t_0,t_1;X)\rightarrow W^{\alpha,1}_{RL}(t_0,t_1;X)$ is a bounded linear operator.
\end{theorem}
\begin{proof}
Observe that the assertion $(ii)$ in Remark \ref{remarkparaajud} guarantees that
\begin{equation*}\dfrac{d^k}{dt^k}\left[J_{t_0,t}^\alpha f(t)\right]=J_{t_0,t}^{\alpha-k}f(t),\end{equation*}
is valid for every $k\in\{0,1,\ldots,[\alpha]-1\}$. Therefore, since item $(i)$ of Remark \ref{remarkfracriint} ensures that $J_{t_0,t}^{\beta}:L^1(t_0,t_1;X)\rightarrow L^1(t_0,t_1;X)$ is bounded for any $\beta>0$, it follows that $J_{t_0,t}^\alpha f(t)\in W^{[\alpha]-1,1}(t_0,t_1;X)$. Again, from item $(ii)$ of Remark \ref{remarkparaajud}, note that
$$D_{t_0,t}^\alpha\left[J_{t_0,t}^{\alpha}f(t)\right]=f(t)\in L^1(t_0,t_1;X).$$
Therefore we have that $J_{t_0,t}^{\alpha}f\in W^{\alpha,1}_{RL}(t_0,t_1;X)$.

Finally, from item $(i)$ of Remark \ref{remarkfracriint} we have that
\begin{multline*}\|J_{t_0,t}^{\alpha}f\|_{W^{\alpha,1}_{RL}(t_0,t_1;X)}=\sum_{k=0}^{[\alpha]-1}\left\|\dfrac{d^k}{dt^k}\left[J_{t_0,t}^\alpha f\right]\right\|_{L^1(t_0,t_1;X)}+\left\|D_{t_0,t}^\alpha\left[J_{t_0,t}^\alpha f\right]\right\|_{L^1(t_0,t_1;X)}\\
=\sum_{k=0}^{[\alpha]-1}\left\|J_{t_0,t}^{\alpha-k} f\right\|_{L^1(t_0,t_1;X)}+\left\|f\right\|_{L^1(t_0,t_1;X)}\leq\left[1+\sum_{k=0}^{[\alpha]-1}\dfrac{(t_1-t_0)^{\alpha-k}}{\Gamma(\alpha-k+1)}\right]\left\|f\right\|_{L^1(t_0,t_1;X)}.\end{multline*}

In other words, $J_{t_0,t}^\alpha:L^1(t_0,t_1;X)\rightarrow W^{\alpha,1}_{RL}(t_0,t_1;X)$ is a bounded linear operator.
\end{proof}

We conclude this section by showing that the space $W^{\alpha,1}_{RL}(t_0,t_1;X)$ is sharp in the following sense:

\begin{theorem} If $\alpha\in[1,\infty)$, $\eta_1\in[\alpha,\infty)$ and $\eta_2\in[1,\infty)$ are such that $\eta_1+\eta_2>\alpha+1$, then $J_{t_0,t}^\alpha\big(L^1(t_0,t_1;X)\big)\not\subset W^{\eta_1,\eta_2}_{RL}(t_0,t_1;X)$.
\end{theorem}

\begin{proof} Assume, without loss of generality, that $t_0=0$, $t_1=1$ and $X=\mathbb{R}$.
\begin{itemize}
\item[(i)] Observe that $J_{0,t}^\alpha\big(L^1(0,1;\mathbb{R})\big)\not\subset W^{\alpha,\eta_2}_{RL}(0,1;\mathbb{R})$ if $\eta_2>1$, since  function $f(t)=t^{\gamma}$, for $\gamma=-(\eta_2+1)/(2\eta_2)$, belongs to $L^1(0,1;\mathbb{R})$ however $D_{0,t}^{\alpha}\left[J_{0,t}^\alpha f(t)\right]=t^\gamma\not\in L^{\eta_2}(0,1;\mathbb{R}).$\vspace*{0.2cm}

%

\item[(ii)] If $\alpha<\eta_1\leq 2$ and $f(t)=t^{\gamma}$, with $-1<\gamma<\eta_1-\alpha-1,$
then:
\begin{itemize}
\item[(a)] $f\in L^1(0,1;\mathbb{R})$, since $\gamma+1>0$;\vspace*{0.2cm}
\item[(b)] Item $(i)$ of Remark \ref{remarkfracriint} ensures that $J_{0,t}^{\alpha-1} f\in L^1(0,1;\mathbb{R})$. Thus, item $(ii)$ of Remark \ref{remarkparaajud} ensures that $(d/dt)J_{0,t}^\alpha f=J_{0,t}^{\alpha-1} f$, that is, $(d/dt)J_{0,t}^\alpha f\in L^1(0,1;\mathbb{R})$;\vspace*{0.2cm}
\item[(c)] Since $0<  1 - \eta_1 + \alpha<\gamma + 2 - \eta_1 + \alpha < 1$, we have
$$\hspace*{1cm}D_{0,t}^{\eta_1}\big[J_{0,t}^\alpha f\big]=(d^{2}/dt^{2})\big[J_{0,t}^{2-\eta_1+\alpha}t^\gamma\big]=M_1(d^{2}/dt^{2})t^{\gamma+2-\eta_1+\alpha}=M_2\,t^{\gamma-\eta_1+\alpha},$$
for some $M_1,M_2>0$. Therefore we have that $D_{0,t}^{\eta_1}\big[J_{0,t}^\alpha f\big]\not\in L^1(0,1;\mathbb{R})$, since $1+\gamma-\eta_1+\alpha<0$.
\end{itemize}
In other words, $J_{0,t}^\alpha\big(L^1(0,1;\mathbb{R})\big)\not\subset W^{\eta_1,1}_{RL}(0,1;\mathbb{R})$.\vspace*{0.2cm}
\item[(iii)] If $\max\{\alpha,2\}<\eta_1$, then $[\eta_1]-1\geq2$ and therefore $W^{\eta_1,1}_{RL}(0,1;\mathbb{R})\subset W^{2,1}(0,1;\mathbb{R})$. Hence, $J_{0,t}^\alpha\big(L^1(0,1;\mathbb{R})\big)\not\subset W^{\eta_1,1}_{RL}(0,1;\mathbb{R})$, since otherwise $J_{0,t}^\alpha\big(L^1(0,1;\mathbb{R})\big)\subset W^{2,1}_{RL}(0,1;\mathbb{R})$, which would contradicts item $(ii)$.\vspace*{0.2cm}
\item[(iv)] Finally, for $\alpha<\eta_1$ and $1<\eta_2$, since  $W^{\eta_1,\eta_2}_{RL}(0,1;\mathbb{R})\subset W^{\eta_1,1}_{RL}(0,1;\mathbb{R}),$
we cannot have $J_{0,t}^\alpha\big(L^1(0,1;\mathbb{R})\big)\subset W^{\eta_1,\eta_2}_{RL}(0,1;\mathbb{R})$, since otherwise $J_{0,t}^\alpha\big(L^1(0,1;\mathbb{R})\big)\subset W^{\eta_1,1}_{RL}(0,1;\mathbb{R})$, which would contradicts items $(ii)$ or  $(iii)$.
\end{itemize}
\end{proof}

\section{Some Special Spaces}

In this section, we present the classical BMO space and introduce the $K_{\gamma}(t_0, t_1; X)$ space. To justify studying the operator $J^\alpha$ from $L^p$ into these spaces, we first need to understand their relation to the $L^p$ spaces. This is our main objective here.

\subsection{The BMO space} We begin with the spaces introduced by F. John and L. Nirenberg in \cite{JoNi1}, here adapted to vector-valued functions.

\begin{definition} If $f\in L^1_{loc}(t_0,t_1;X)$ and
$$[f]_{BMO(t_0,t_1;X)}:=\sup_{[a,b]\subset (t_0,t_1)}{\left[\dfrac{1}{(b-a)}\int_a^b{\|f(s)-\operatorname{avg}_{[a,b]}(f)\|_X}\,ds\right]}<\infty,$$
where
$$\operatorname{avg}_{[a,b]}(f):=\dfrac{1}{(b-a)}\int_a^b{f(s)}\,ds,$$
we say that $f$ is a vector-valued function whose mean oscillation is bounded (finite). We denote the vector space of all those functions by $BMO(t_0,t_1;X)$.
\end{definition}

Functions in $BMO(t_0,t_1;X)$ are important in harmonic analysis and partial differential equations, occasionally serving as a substitute for the space $L^\infty(t_0,t_1;X)$. Not being different here, this space ends up demonstrating great importance in our discussion.

\begin{remark}
 Notice that $[\,\cdot\,]_{BMO(t_0,t_1;X)}$ does not define a norm in $BMO(t_0,t_1;X)$, since for any vector $x\in X$ we have that
$$[x]_{BMO(t_0,t_1;X)}=0.$$
\end{remark}

\begin{proposition}\label{propaux04} It holds that $L^\infty(t_0,t_1;X) \subsetneq BMO(t_0,t_1;X)$
\end{proposition}

\begin{proof} The definition of $[\,\cdot\,]_{BMO(t_0,t_1;X)}$ readily implies, for any $f\in L^\infty(t_0,t_1;X)$, the inequality
    $$[f]_{BMO(t_0,t_1;X)}\leq2\|f\|_{L^\infty(t_0,t_1;X)},$$
    establishing the inclusion $L^\infty(t_0,t_1;X)\subset BMO(t_0,t_1;X)$. To prove the strictness of the inclusion, we provide an adaptation of \cite{Mo1}. Suppose that $t_0=0$ and $t_1=1$. Consider $x\in X\setminus\{0\}$ and $\phi:(0,1)\rightarrow X$ given by $\phi(t)=\big(\log{t}\big)x$. Clearly $\phi\not\in L^\infty(0,1;X)$. On the other hand, for $\varepsilon>1$ and $[a,a\varepsilon]\subset (0,1)$, we have that
        \begin{multline*}\hspace*{2.3cm}\dfrac{\|x\|_X}{(a\varepsilon-a)}\int_a^{a\varepsilon}{\big|\phi(s)-\operatorname{avg}_{[a,a\varepsilon]}(\phi)\big|}\,ds
        \\=\|x\|_X\underbrace{\left(\dfrac{2}{e}\right)\left[\dfrac{(\varepsilon-1)\varepsilon^{\varepsilon/(\varepsilon-1)}-e[\varepsilon\log\varepsilon
    ]}{(\varepsilon-1)^2}\right]}_{=\eta(\varepsilon)}.\end{multline*}
     Observing that function $\eta(\varepsilon)$ is both monotonically increasing and continuous in $(1,\infty)$, it follows that
    \begin{multline*}\hspace*{2.3cm}[\phi]_{BMO(0,1;X)}=\|x\|_X\sup{\{\eta(\varepsilon):\varepsilon>1\textrm{ such that }[a,a\varepsilon]\subset(0,1)\}}\\=\|x\|_X\lim_{\epsilon\rightarrow\infty}\eta(\varepsilon)=(2/e)\|x\|_X<\infty.\vspace*{0.2cm}\end{multline*}
\end{proof}

Since $L^\infty(t_0, t_1; X)\subset L^p(t_0, t_1; X)$, a natural question that arises is the relation between BMO spaces and $L^p$ spaces.

\begin{remark}
As demonstrated by the Johnson-Neugebauer example (refer to \cite[Remark (2) of Theorem 3]{JoNe1} for more details), which is given by
    $$\phi(t)=\max\{\log(1/t),1/(\log(te^2))\},$$
    we can observe that $\phi\in BMO(0,1;\mathbb{R})$ but $\phi\not\in L^p(0,1;\mathbb{R})$. Therefore, we obtain $BMO(t_0, t_1; X)\nsubseteq L^p(t_0, t_1; X)$, for $1\leq p<\infty$.
\end{remark}

Despite the previous Remark, if we consider $L_{\text{loc}}^r(t_0, t_1; X)$ we have:

\begin{proposition}
It holds that $BMO(t_0, t_1; X)\subset\bigcap_{r\geq1}L_{\text{loc}}^r(t_0, t_1; X).$
\end{proposition}
\begin{proof} For any interval $[a, b] \subset (t_0, t_1)$ and $f \in BMO(t_0, t_1; X)$, it holds that:
$$\big|\{s\in[a,b]:\|f(s)-\operatorname{avg}_{[a,b]}(f)\|_X>\lambda\}\big|\leq c_1(b-a)e^{-c_2\lambda/\|f\|_{BMO(t_0,t_1;X)}},$$
for $c_1,c_2> 0$ (see \cite{JoNi1}). From this inequality, we can prove that
\begin{multline}\label{helpme00}\hspace*{1cm}\int_a^b\|f(s)-\operatorname{avg}_{[a,b]}(f)\|^r_X\,ds\leq(t_1-t_0)c_1r\int_0^\infty s^{r-1}e^{-c_2s/\|f\|_{BMO(t_0,t_1;X)}}\,ds\\=(t_1-t_0)c_1\left(\dfrac{\|f\|_{BMO(t_0,t_1;X)}}{c_2}\right)^r\Gamma(r+1),\end{multline}
for any $r\geq1$. Furthermore, for any compact set $K \subset (t_0, t_1)$, there exists an interval $[a, b] \subset (t_0, t_1)$ such that $K \subset [a, b].$ Consequently, for $r\geq1$, we can conclude:
\begin{multline*}\hspace*{1cm}\|f\|_{L^r(K;X)}\leq\|f-\operatorname{avg}_{[a,b]}(f)\|_{L^r(a,b;X)}+\|\operatorname{avg}_{[a,b]}(f)\|_{L^r(a,b;X)}
\\\leq \dfrac{\big[(t_1-t_0)c_1\Gamma(r+1)\big]^{1/r}\|f\|_{BMO(t_0,t_1;X)}}{c_2}+(b-a)^{(1/r)-1}\|f\|_{L^1(a,b;X)}<\infty,\end{multline*}
This implies that $BMO(t_0, t_1; X)\subset\bigcap_{r\geq1}L_{\text{loc}}^r(t_0, t_1; X).$
\end{proof}

To relate $BMO(t_0, t_1; X)$ and $L^p(t_0, t_1; X)$ spaces, we have:

\begin{theorem}
It holds that $BMO(t_0,t_1;X)\cap L^1(t_0,t_1;X)\subsetneq \bigcap_{r\geq1}L^r(t_0,t_1;X)$. Moreover, for any $r\geq1$ and $f\in BMO(t_0,t_1;X)\cap L^1(t_0,t_1;X),$ there exists $C>0$ such that
\begin{equation}\label{iterbmo}\|f\|_{L^r(t_0,t_1;X)}\leq Cr\|f\|^{1/r}_{L^1(t_0,t_1;X)}\|f\|^{1-(1/r)}_{BMO(t_0,t_1;X)}.\end{equation}
\end{theorem}

 \begin{proof} The proof of the inclusion and the inequality are simply adaptations of a classical result (see \cite[Theorem 2]{ChZh1}).

 Now to prove the strictness of the inclusion, assume without loss of generality that $t_0=0$, $t_1=1$ and $x\in X\setminus\{0\}$. Then consider function $\psi:(0,1)\rightarrow X$ given by
$$\psi(t) = \big[\chi_{(1/2,1]}(t)\log\big(1/(t-(1/2))\big)\big]x,$$
where $\chi_{(1/2,1]}(t)$ denotes the characteristic function of $(1/2,1]$. The proof that $\psi\in L^r(0,1;X)$ for every $r\geq1$ is direct, and the proof that $\psi\not\in BMO(0,1;X)$ follows the same arguments as in \cite[Example 7.1.4]{Gr1}.
\end{proof}

\begin{remark}\label{coraux001} We can summarize the above results as:
$$L^\infty(t_0,t_1;X)\subsetneq BMO(0,1;X)\cap L^1(t_0,t_1;X)\subsetneq\bigcap_{r\geq1}L^r(t_0,t_1;X).$$
\end{remark}

\subsection{The $K_{\gamma}(t_0, t_1; X)$ space}

Now we introduce a vector-valued version of a function space, which was approached by Karapetyants and Rubin in \cite{KaRu1,KaRu2}. These are the oldest references we have found.

\begin{definition} For $\gamma > 0$, the Karapetyants-Rubin space (KR-space, for short) is the collection of functions $f \in \bigcap_{r \geq 1} L^r(t_0, t_1; X)$ that satisfy
$$\sup\{r^{-\gamma}\|f\|_{L^r(t_0,t_1;X)}:r\geq1\}<\infty.$$
We denote this space by $K_{\gamma}(t_0, t_1; X)$.
\end{definition}

From the above definition, several noteworthy aspects become apparent. It is crucial to emphasize that all the subsequent proofs and conclusions, to the best of the author's knowledge, have not been previously proved in the literature.

\begin{theorem}\label{inclurubin} In $K_{\gamma}(t_0, t_1; X)$, if we consider
$$\|f\|_{K_{\gamma}(t_0, t_1; X)}:=\sup\{r^{-\gamma}\|f\|_{L^r(t_0,t_1;X)}:r\geq1\},$$
then it becomes a Banach space.
\end{theorem}

\begin{proof} It is not difficult to conclude that $K_{\gamma}(t_0, t_1; X)$ is a normed vector space. Now consider the Cauchy sequence $\{f_j\}_{j=1}^\infty \subset K_\gamma(t_0, t_1; X)$. Observe that for any $j\in\mathbb{N}$ and $r\geq1$ we have
    \begin{equation*}\|f_j\|_{L^r(t_0,t_1;X)}=r^\gamma \big[r^{-\gamma}\|f_j\|_{L^r(t_0,t_1;X)}\big]\leq r^\gamma\|f_j\|_{K_{\gamma}(t_0, t_1; X)},\end{equation*}
    what allows us to deduce that $\{f_j\}_{j=1}^\infty$ is a Cauchy sequence in $L^r(t_0,t_1;X)$. Due to the arbitrary nature of $r\geq1$ and the boundedness of the interval $(t_0,t_1)$,  we deduce the existence of $f\in\bigcap_{r \geq 1} L^r(t_0, t_1; X)$ such that $f_j\rightarrow f$ in $L^r(t_0,t_1;X)$, for each $r\geq1$. Consequently, we have that
    $$\lim_{j\rightarrow\infty}r^{-\gamma}\|f_j\|_{L^r(t_0,t_1;X)}= r^{-\gamma}\|f\|_{L^r(t_0,t_1;X)},$$
    for each $r \geq 1$. On the other hand, since $\{f_j\}_{j=1}^\infty$ is a Cauchy sequence in $K_\gamma(t_0, t_1; X)$ and
     $$\|f_j\|_{K_{\gamma}(t_0, t_1; X)}=\sup\{r^{-\gamma}\|f_j\|_{L^r(t_0,t_1;X)}:r\geq1\},$$
     we deduce that $\{r^{-\gamma}\|f_j\|_{L^r(t_0,t_1;X)}\}_{j=1}^\infty$ is a Cauchy sequence in $B([1,\infty);\mathbb{R})$. Therefore, there exists $M\in B([1,\infty);\mathbb{R})$ such that
     $$\lim_{j\rightarrow\infty}\Big[\sup\big\{\big|r^{-\gamma}\|f_j\|_{L^r(t_0,t_1;X)}-M(r)\big|:r\geq1\big\}\Big]=0.$$
     But then
    $$\lim_{j\rightarrow\infty}r^{-\gamma}\|f_j\|_{L^r(t_0,t_1;X)}= M(r),$$
     for each $r \geq 1$, and therefore $M(r)=r^{-\gamma}\|f\|_{L^r(t_0,t_1;X)}$, what completes the proof.
\end{proof}

\begin{theorem}\label{incluaux001} For $0 < \gamma_1 < \gamma_2$, we have that $K_{\gamma_1}(t_0, t_1; X) \subsetneq K_{\gamma_2}(t_0, t_1; X)$ continuously.
\end{theorem}

\begin{proof} The inclusion of spaces is proven straightforwardly, and the strictness of this inclusion is a direct consequence of the following Lemmas.
\end{proof}

 It is important to note that, in the following results, we assume that $\psi_0(z)$ and $\psi_1(z)$ represent the analytical Digamma and Trigamma functions, respectively. For more comprehensive information on these functions, we refer to \cite{AbSt1}.

\begin{lemma}\label{lemmaaux001} For $\zeta\geq1$, consider $\rho_\zeta:[1,\infty)\rightarrow\mathbb{R}$ given by
$$\rho_\zeta(s)=\log\big(\Gamma(s)\big)-(s-1)\big[\psi_{0}(s)-\zeta\big].$$
Then $\rho_\zeta(s)>0$ for every $t>1$.
\end{lemma}

\begin{proof} Observe that $\rho_\zeta^\prime(s) = \zeta - (s-1)\psi_1(s)$. Now, since Guo-Qi established in \cite[Lemma 2]{GuQi1} that $\psi_1(s) < e^{1/s} - 1$ for every $s > 1$, we can straightforwardly deduce that $\psi_1(s) < 1/(s-1)$ for every $s > 1$. Therefore, for any $\zeta\geq1$ we have that $\psi_1(s) < \zeta/(s-1)$ for every $s > 1$, what allows us to conclude that $\rho_\zeta(s)$ strictly increases in $(1, \infty)$. Given that $\rho_\zeta(1) = 0$, we conclude that $\rho_\zeta(s) > 0$ for all $s > 1$.
\end{proof}

\begin{lemma}\label{lemmaaux002} For $x\in X$ and $\gamma>0$, define function $\phi_\gamma:(0,1)\rightarrow X$ by
$$\phi_\gamma(t)=\big[\log{(1/t)}\big]^\gamma x.$$
Then $\phi_\gamma\in \bigcap_{r \geq 1} L^r(t_0, t_1; X)$ and
$$\|\phi_\gamma\|_{K_{\sigma}(0, 1;X)}=\left\{\begin{array}{ll}\Gamma(\gamma+1),&\textrm{ if }\sigma\geq\gamma,\\
\infty,& \textrm{ if }0<\sigma<\gamma.\end{array}\right.$$
\end{lemma}
\begin{proof}  A direct computation shows that $\phi_\gamma\in \bigcap_{r \geq 1} L^r(t_0, t_1; X)$. For $\sigma>0$, consider function $\upsilon_\sigma:[1,\infty)\rightarrow\mathbb{R}$, given by $\upsilon_\sigma(s)=s^{-\sigma}\big[\Gamma(s\gamma+1)\big]^{1/s}$.

When $\sigma\geq\gamma$, since $\upsilon_\sigma^\prime(s)=-s^{-\sigma-2}\big[\Gamma(s\gamma+1)\big]^{1/s}\rho_{\sigma/\gamma}(s\gamma +1)$ (where $\rho$ is given in Lemma \ref{lemmaaux001}), Lemma \ref{lemmaaux001} ensures that $\upsilon_\sigma(s)$ is decreasing in $(1,\infty)$. From this we may verify that
         $$\|\phi_\gamma\|_{K_{\sigma}(0, 1;X)}=\|x\|_X\sup\{r^{-\sigma}\big[\Gamma(r\gamma+1)\big]^{1/r}:r\geq1\}=\Gamma(\gamma+1)<\infty.$$

         On the other hand, when $0<\sigma<\gamma$, the lower bound of the Gamma function (for details we refer to \cite{Ja1}) allows us to conclude that there exists $C > 0$ such that
        \begin{equation*}\hspace*{1,1cm}\|\phi_{\gamma}\|_{K_{\gamma}(0, 1;X)}\geq s^{-\sigma}\big[\Gamma(\gamma s+1)\big]^{1/s}\|x\|\geq Cs^{\gamma-\sigma},\end{equation*}
        for sufficiently large values of $s\geq1$. Therefore $\|\phi\|_{K_{\gamma}(0, 1;X)}=\infty.$
\end{proof}

Since, by definition, $K_{\gamma}(t_0, t_1; X) \subset L^p(t_0, t_1; X)$, a natural question that arises is the relation with $L^\infty(t_0,t_1;X)$.

\begin{theorem} For all $\gamma > 0$, $L^\infty(t_0,t_1;X)\subsetneq K_{\gamma}(t_0, t_1; X)$.
\end{theorem}

\begin{proof} The inclusion of the spaces is straightforwardly proven. Now the strictness of this inclusion is a direct consequence of Lemma \ref{lemmaaux002}.
\end{proof}

To conclude our discussion in this subsection, let us now delve into the relationship between the spaces $BMO(t_0, t_1; X)\cap L^1(t_0, t_1; X)$ and $K_\gamma(t_0, t_1; X)$. It is worth noting that both of these function spaces contains properly $L^\infty(t_0, t_1; X)$ and are properly contained in $\bigcap_{r \geq 1} L^r(t_0, t_1; X)$ (cf. Corollary \ref{coraux001}).

\begin{theorem}\label{auxtheo001} For $\gamma\geq1$, $BMO(t_0, t_1; X)\cap L^1(t_0, t_1; X)\subsetneq K_\gamma(t_0, t_1; X)$.
\end{theorem}

\begin{proof} Let $f\in BMO(t_0, t_1; X)\cap L^1(t_0, t_1; X)$ and $r\geq1$. Then \eqref{iterbmo} ensures that
$$r^{-\gamma}\|f\|_{L^r(t_0,t_1;X)}\leq Cr^{1-\gamma}\|f\|^{1/r}_{L^1(t_0,t_1;X)}\|f\|^{1-(1/r)}_{BMO(t_0,t_1;X)},$$
which allow us to deduce that
$$\sup\{r^{-\gamma}\|f\|_{L^r(t_0,t_1;X)}:r\geq1\}\leq C(1+\|f\|_{L^1(t_0,t_1;X)})(1+\|f\|_{BMO(t_0,t_1;X)})<\infty,$$
which implies that $f\in K_\gamma(t_0, t_1; X)$.

To prove that $BMO(t_0, t_1; X)\cap L^1(t_0, t_1; X)$ is strictly contained in $K_\gamma(t_0, t_1; X)$, let us assume without loss of generality that $t_0=0$ and $t_1=1$. Now let $x\in X$ and define $\psi_\gamma:[0,1]\rightarrow\mathbb{R}$ as follows:
$$\psi_\gamma(t)= \big[\chi_{(1/2,1]}(t)\log\big(1/(t-(1/2))\big)\big]^\gamma x,$$
where $\chi_{(1/2,1]}(t)$ denotes the characteristic function of $(1/2,1]$. The proof that $\psi_\gamma\in K_\gamma(0,1;X)$ follows from Lemma \ref{lemmaaux002}, and the proof that $\psi_\gamma\not\in BMO(0,1;X)$ is a simple adaptation of \cite[Example 7.1.4]{Gr1}.
\end{proof}

\begin{corollary} For $0 < \gamma < 1$, it holds that constant functions belong to $K_\gamma(t_0, t_1; X)$ and $BMO(t_0, t_1; X)\cap L^1(t_0,t_1;X)$. However, $BMO(t_0, t_1; X)\cap L^1(t_0,t_1;X)$ is not a subset of $K_\gamma(t_0, t_1; X),$ and similarly, $K_\gamma(t_0, t_1; X)$ is not a subset of $BMO(t_0, t_1; X)\cap L^1(t_0,t_1;X)$.
\end{corollary}

\begin{proof} Let us assume $t_0=0$ and $t_1=1$. Consider an arbitrary element $x\in X\setminus\{0\}$ and recall function $\psi_\gamma:[0,1]\rightarrow\mathbb{R}$ defined as:
$$\psi_\gamma(t)= \big[\chi_{(1/2,1]}(t)\log\big(1/(t-(1/2))\big)\big]^\gamma x.$$
We already know that $\psi_\gamma\not\in BMO(0,1;X)\cap L^1(t_0,t_1;X)$ (see \cite[Example 7.1.4]{Gr1}). On the other hand, proving that $\psi_\gamma\in K_\gamma(t_0,t_1;X)$ is a consequence of Lemma \ref{lemmaaux002}.

Finally, if we consider $\phi(t)=\log(t)$, then $\phi\in BMO(0,1;X)\cap L^1(0,1;X)$ (just recall Proposition \ref{propaux04}). On the other hand, it follows from Lemma \ref{lemmaaux002} that $\phi\not\in K_\gamma(0,1;X)$, for $0<\gamma<1$.
\end{proof}

\begin{theorem} For $0<\gamma<1$ we have that $BMO(t_0, t_1; X)\cap K_\gamma(t_0, t_1; X)$ is a Banach space when considered with the norm
$$\|\phi\|_\gamma:=\|\phi\|_{BMO(t_0, t_1; X)}+\|\phi\|_{K_\gamma(t_0, t_1; X)}.$$
\end{theorem}

\begin{proof} A direct computation shows that $\|\cdot\|_\gamma$ defines a norm. Let us assume that $\{\phi_j\}_{j=1}^\infty \subset BMO(t_0, t_1; X) \cap K_\gamma(t_0, t_1; X)$ is a Cauchy sequence with respect to the norm $\|\cdot\|_\gamma$. Then, it follows from Theorem \ref{inclurubin} that there exists $f \in K_\gamma(t_0, t_1; X)$ such that $\phi_j \rightarrow f$ with respect to the norm $\|\cdot\|_{K_\gamma(t_0, t_1; X)}$. Now, since we have
\begin{equation*}\|\phi_j-f\|_{L^1(t_0,t_1;X)}\leq \|\phi_j-f\|_{K_\gamma(t_0, t_1; X)},\end{equation*}
we can deduce that the convergence $\phi_j \rightarrow f$ in $L^1(t_0,t_1;X)$ implies that there exists $\{\phi_{j_k}\}_{k=1}^\infty \subset \{\phi_{j}\}_{j=1}^\infty$ such that $\phi_{j_k}(t) \rightarrow f(t)$ for almost every $t \in (t_0,t_1)$.

Now, by recalling \cite[Remarks 1.5 and 1.6]{Ne1}, we know that $BMO(t_0, t_1; X)$ modulo the space of constant functions is a Banach space. Therefore, since $\{\phi_{j_k}\}_{k=1}^\infty$ is a Cauchy sequence with respect to $\|\cdot\|_{BMO(t_0, t_1; X)}$, there exists $g\in BMO(t_0, t_1; X)$ such that $\phi_{j_k} \rightarrow g$ with respect to $\|\cdot\|_{BMO(t_0, t_1; X)}$. Then, if we take $[a,b]\subset(t_0,t_1)$, thanks to \eqref{helpme00}, there exists $\{\phi_{j_{k_l}}\}_{l=1}^\infty\subset\{\phi_{j_k}\}_{k=1}^\infty$ such that
\begin{equation}\label{novaequ001}\phi_{j_{k_l}}(t)-\operatorname{avg}_{[a,b]}(\phi_{j_{k_l}})\rightarrow g(t)-\operatorname{avg}_{[a,b]}(g)\end{equation}
for almost every $t\in[a,b]$. Therefore, \eqref{novaequ001}, the fact that $\phi_{j_k} \rightarrow f$ in $L^1(t_0,t_1;X)$ and also $\phi_{j_k}(t) \rightarrow f(t)$ for almost every $t \in (t_0,t_1)$, ensure that
$$f(t)-\operatorname{avg}_{[a,b]}(f)=g(t)-\operatorname{avg}_{[a,b]}(g)$$
for almost every $t\in[a,b]$. Since $[a,b]$ was an arbitrary closed interval contained in $(t_0,t_1)$, we can conclude that $\phi_{j_k} \rightarrow f$ with respect to $\|\cdot\|_{BMO(t_0, t_1; X)}$. From this last argument we may infer that $\phi_{j_k} \rightarrow f$ with respect to $\|\cdot\|_\gamma$, thereby establishing that $\phi_{j} \rightarrow f$ with respect to $\|\cdot\|_\gamma$, as we wanted.
\end{proof}

\section{Boundeness of the RL Fractional Integral - The Critical Case}\label{subboa}

In this section, we address the continuity of the RL fractional integral operator of order $\alpha=1/p$ in $L^p(t_0, t_1; X)$, specifically when $p>1$. Since the existing literature on this subject is sparse and incomplete, we dedicate this section to shed as much light as possible on this topic.

Notably, for the Riesz fractional integral $I^\alpha$, Stein and Zygmund in \cite{StZy1} have already established the continuity of $I^\alpha$ from $L^p(\mathbb{R}^n)$ into $BMO(\mathbb{R}^n)$. Other contributors to this line of study include Muckenhoupt-Wheeden in \cite{MuWh1}, Samko in \cite{Sa1} and Rafeiro-Samko in \cite{SaRa1}.

On the other hand, when considering the Riemann-Liouville fractional integral, references become more restricted and obscure. The primary reference (known to the authors) on this topic is the classical book \cite{EdKoMe1} by Edmunds, Kokilashvili, and Meskhi. In Section 2.15 of their book, the authors mention that the proof of the continuity of the RL fractional integral with order $\alpha=1/p$ from $L^p(t_0, t_1; \mathbb{R})$ into $K_{\gamma}(t_0, t_1; \mathbb{R})$, is given in \cite{KaRu1}. However, despite thorough investigation, we were unable to locate this particular article.

Nonetheless, we did find a work, \cite{KaRu2}, presenting statements akin to those mentioned by Edmunds, Kokilashvili, and Meskhi in their book and authored by the same individuals they cite; namely, N. K. Karapetiantz and B. S. Rubin. This work could serve as an official reference on this topic. It is essential to note that, while this source describes some intriguing results, it does not provide any proofs.

With that in mind, in what follows, we not only establish some of the results presented by Karapetiantz and Rubin in their work \cite{KaRu2}, extending them to vector-valued functions, but we also refine these findings to derive a novel result concerning the continuity of the RL fractional integral.

With that in mind, in what follows, we not only proof some of the results presented by Karapetiantz and Rubin in their work [19] but also provide a completely new approach compared to the one suggested by them (interpolation results from \cite[Theorem 5]{CarFe1}). Furthermore, we extend these results to vector-valued functions and refine them to derive a novel result concerning the continuity of the RL fractional integral.

\begin{theorem}\label{HL01} For $p>1$ and $q\geq1$, it holds that
\begin{equation}\label{neweq0-1}\|J_{t_0,t}^{1/p}f\|_{L^q(t_0, t_1; X)}\leq C(p,q)(t_1-t_0)^{1/q}\|f\|_{L^p(t_0, t_1; X)},\end{equation}
for every $f\in L^p(t_0, t_1; X)$, where
\begin{equation}\label{neweq00}C(p,q)=\dfrac{2p^{(p+(p/q)-(1/q))/(p+q-1)}(q-1)^{(q-1)(1-(1/p))/(p+q-1)}}{\Gamma(1/p)}.\end{equation}
Moreover, it holds that $J_{t_0,t}^{1/p}\big(L^p(t_0,t_1;X)\big)\not\subset L^\infty(t_0,t_1;X)$.
\end{theorem}

\begin{proof} Let $f\in L^p(t_0, t_1; X)$. We observe that \cite[{item $(iii)$ of Remark 4}]{CarFe1} ensures
\begin{equation}\label{neweq1}\left\|J^{1/p}_{t_0,t}f\right\|_{L^q(t_0,t_1;X)}\leq p^{1/q}(t_1-t_0)^{1/pq}\left[J^{1/p}_{t_0,t}f\right]_{L_w^{pq/(p-1)}(t_0,t_1;X)},\end{equation}
where ${L_w^{pq/(p-1)}(t_0,t_1;X)}$ denotes the weak Lebesgue space (see \cite[{item $(iv)$ of Remark 4}]{CarFe1} for details). Also note that \cite[{Corollary 1}]{CarFe1} ensures that
\begin{multline}\label{neweq2}\left[J^{1/p}_{t_0,t}f\right]_{L_w^{pq/(p-1)}(t_0,t_1;X)}\\\leq \left[\dfrac{2p^{(p-1)/(p+q-1)}(q-1)^{(q-1)(1-(1/p))/(p+q-1)}}{\Gamma(1/p)}\right]\|f\|_{L^{pq/(p+q-1)}(t_0,t_1;X)}.\end{multline}
Hence, \eqref{neweq1}, \eqref{neweq2}, and the continuous inclusion of $L^{p}(t_0,t_1;X)$ in $L^{pq/(p+q-1)}(t_0,t_1;X)$ show that
 \eqref{neweq0-1} holds.

The proof that $J_{t_0,t}^{1/p}$ does not maps the space $L^{p}(t_0,t_1;X)$ into $L^\infty(t_0,t_1;X)$ is an adaptation of an argument by Hardy and Littlewood; see \cite[p. 578]{HaLi1} for details.
\end{proof}

The last theorem demonstrated that the RL fractional integral of order $1/p$ maps $L^p(t_0,t_1;X)$ into $\cap_{r\geq1}L^r(t_0,t_1;X)$ but not into $L^\infty(t_0,t_1;X)$. This leads us to inquire about a Banach space strictly between $L^\infty(t_0,t_1;X)$ and $\cap_{r\geq1}L^r(t_0,t_1;X)$ and where we can establish that the RL fractional integral of order $1/p$ indeed maps $L^p(t_0,t_1;X)$ into this space.

\begin{theorem}\label{final011} For $p>1$, it holds that $J^{1/p}_{t_0,t}:L^p(t_0, t_1; X)\rightarrow BMO(t_0, t_1; X)$ and that
$$[J_{t_0,t}^{1/p}f]_{BMO(t_0, t_1; X)}\leq\dfrac{4}{\Gamma\big((1/p)+1\big)}\|f\|_{L^p(t_0, t_1; X)},$$
for every $f\in L^p(t_0, t_1; X)$.
\end{theorem}

\begin{proof} Let $[a,b]\subset(t_0,t_1)$ and notice that
\begin{multline}\label{auxmul001}\left[\dfrac{1}{b-a}\right]\int_a^b\left\|J_{t_0,t}^{1/p}f(t)-\operatorname{avg}_{[a,b]}(J_{t_0,s}^{1/p}f)\right\|_X\,dt
\\\leq\left[\dfrac{1}{(b-a)^2}\right]\int_a^b\int_a^b\left\|J_{t_0,t}^{1/p}f(t)-J_{t_0,s}^{1/p}f(s)\right\|_X\,ds\,dt
\leq\left[\dfrac{2}{(b-a)}\right]\int_a^b\left\|J_{t_0,t}^{1/p}f(t)\right\|_X\,dt.\end{multline}

Now observe that
\begin{multline}\label{auxmul002}\int_a^b\left\|\dfrac{1}{\Gamma(1/p)}\int_{t_0}^t(t-s)^{(1/p)-1}f(s)\,ds\right\|_X\,dt\\\leq
\int_a^b\left\|\dfrac{1}{\Gamma(1/p)}\int_{t_0}^a(t-s)^{(1/p)-1}f(s)\,ds\right\|_X\,dt\\+\int_a^b\left\|\dfrac{1}{\Gamma(1/p)}\int_{a}^t(t-s)^{(1/p)-1}f(s)\,ds\right\|_X\,dt=:\mathcal{I}+\mathcal{J}.
\end{multline}

It is not difficult to see that Fubini-Tonelli Theorem, H\"{o}lder's inequality and the continuous inclusion of $L^{p}(t_0,t_1;X)$ into $L^1(t_0,t_1;X)$ ensure that
\begin{multline*}\mathcal{I}\leq\dfrac{1}{\Gamma(1/p)}\int_{t_0}^{a}\left[\int_a^b(t-s)^{(1/p)-1}\,dt\right]\|f(s)\|_X\,ds\\
=\dfrac{1}{\Gamma((1/p)+1)}\int_{t_0}^{a}\left[(b-s)^{(1/p)}-(a-s)^{(1/p)}\right]\|f(s)\|_X\,ds\\\leq
\dfrac{(b-a)^{1/p}}{\Gamma((1/p)+1)}\|f\|_{L^1(a,b;X)}\leq\dfrac{(b-a)}{\Gamma((1/p)+1)}\|f\|_{L^p(a,b;X)}.\end{multline*}

On the other hand, it follows from item $(i)$ of Remark \ref{remarkfracriint} and the continuous inclusion of $L^{p}(t_0,t_1;X)$ into $L^1(t_0,t_1;X)$ that
$$\mathcal{J}\leq\dfrac{(b-a)^{1/p}}{\Gamma((1/p)+1)}\|f\|_{L^1(a,b;X)}\leq\dfrac{(b-a)}{\Gamma((1/p)+1)}\|f\|_{L^p(a,b;X)}.$$

Therefore, \eqref{auxmul001} and \eqref{auxmul002} allow us to deduce that
$$\dfrac{1}{b-a}\int_a^b\left\|J_{t_0,t}^{1/p}f(t)-\operatorname{avg}_{[a,b]}(J_{t_0,s}^{1/p}f)\right\|_X\,dt
\\\leq\dfrac{4}{\Gamma\big((1/p)+1\big)}\|f\|_{L^p(t_0,t_1; X)},$$
as we wanted.

\end{proof}

\begin{theorem}\label{final022} Let $p>1$ and $\gamma\geq(p-1)/p$. Then $J^{1/p}_{t_0,t}:L^p(t_0, t_1; X)\rightarrow K_\gamma(t_0, t_1; X)$ is a bounded linear operator, i.e., there exists $M>0$ such that
$$\|J_{t_0,t}^{1/p}f\|_{K_\gamma(t_0, t_1; X)}\leq M\|f\|_{L^p(t_0, t_1; X)},$$
for every $f\in L^p(t_0, t_1; X)$.
\end{theorem}

\begin{proof} Consider $h:(1,\infty)\rightarrow\mathbb{R}$ given by
$$h(t)=\dfrac{e^{[(1-(1/p))(t-1)/(p+t-1)]\log(t-1)}}{t^{(p-1)/p}},$$
and observe that $h(t)$ is continuous and that
$$\lim_{t\rightarrow1^+}h(t)=1\qquad\textrm{and}\qquad \lim_{t\rightarrow\infty}h(t)=0.$$
Therefore $h(t)$ is bounded in $(1,\infty)$. Hence, since for $f\in L^p(t_0, t_1; X)$, it follows from from \eqref{neweq0-1} and \eqref{neweq00} that
\begin{multline*}r^{-{(p-1)/p}}\|J_{t_0,t}^{1/p}f\|_{L^r(t_0, t_1; X)}\\\leq2p^{(p+(p/r)-(1/r))/(p+r-1)}\left(\dfrac{(t_1-t_0)^{1/r}}{\Gamma(1/p)}\right)h(r)\|f\|_{L^p(t_0, t_1; X)},\end{multline*}
for any $r\geq1$, we deduce the existence of $M\geq0$ such that
$$\|J_{t_0,t}^{1/p}f\|_{K_{(p-1)/p}(t_0, t_1; X)}\leq M\|f\|_{L^p(t_0, t_1; X)},$$
proving the continuity of the RL fractional integral of order $1/p$ from $L^p(t_0, t_1; X)$ into $K_{(p-1)/p}(t_0, t_1; X)$. To obtain the general case we just apply Theorem \ref{incluaux001}.

\end{proof}

Our final and main result can now be stated as a direct consequence of Theorems \ref{final011} and \ref{final022}.

\begin{theorem} Let $p>1$ and $\gamma\geq(p-1)/p$. Then $$J^{1/p}_{t_0,t}:L^p(t_0, t_1; X)\rightarrow BMO(t_0,t_1;X)\cap K_\gamma(t_0, t_1; X)$$ is a bounded linear operator, i.e., there exists $K>0$ such that
$$\|J_{t_0,t}^{1/p}f\|_{K_\gamma(t_0, t_1; X)}\leq K\|f\|_{\gamma},$$
for every $f\in L^p(t_0, t_1; X)$.
\end{theorem}

\section{Open Problems}

An intriguing question that arose during the proof of Theorem \ref{HL01} is the task of identifying a function in $L^p(t_0,t_1;X)$ for which the RL fractional integral of order $1/p$ becomes unbounded. Throughout this study, we have come to appreciate the complexity of this problem, recognizing that our efforts thus far have not produced a solution. However, this is not unexpected, as even Hardy and Littlewood, in their pioneering study \cite{HaLi1}, did not provide such an example. Instead of presenting an explicit example, they resorted to a proof by contradiction to affirmatively address this question. This approach, in some sense, underscores the considerable difficulty in discovering such a function.

Another unresolved issue we encountered pertains to proving that, for any $\gamma\in[0,(p-1)/p)$, it holds that $J_{t_0,t}^{1/p}\big(L^p(t_0,t_1;X)\big)\not\subset K_\gamma(t_0,t_1;X)$. Successfully addressing this would complete our study. However, this problem hinges on the existence of the elusive function mentioned earlier, as it requires a detailed analysis of how the $L^r$ norm of a function changes with varying $r\geq1$, and eventually tends to infinity faster then $r^\gamma$. We acknowledge that the resolution of this particular problem remains elusive.

\section*{Acknowledgement}
 The authors express their gratitude to the Federal University of Esp\'{\i}rito Santo and the Federal University of Santa Catarina for their hospitality and support during short-term visits. The second author was supported by CAPES/PRAPG grant nº 88881.964878/2024-01 and Funda\c{c}\~{a}o de Amparo \`{a} Pesquisa e Inova\c{c}\~{a}o do Esp\'{\i}rito Santo (Fapes) under grant numbers T.O. 427/2023 and T.O. 951/2023.

\end{document}